\documentclass[a4paper]{article}

 \usepackage{a4wide}

 \usepackage{cite}
 \usepackage{authblk}

 \usepackage{graphicx}
 \usepackage{tikz}

 \usepackage{amsmath,amssymb}
 \usepackage{amsthm}
   \theoremstyle{plain}
   \newtheorem{theorem}{Theorem}
   \newtheorem{proposition}[theorem]{Proposition}
   \newtheorem{lemma}[theorem]{Lemma}
   \newtheorem{corollary}[theorem]{Corollary}
   \theoremstyle{definition}
   \newtheorem{definition}[theorem]{Definition}
   \theoremstyle{remark}
   \newtheorem{remark}[theorem]{Remark}
   
 \usepackage{array}
 \usepackage{txfonts}

 \newcommand{\nc}{\newcommand}
 
 \nc{\la}{\lambda}
 \nc{\La}{\Lambda}
 \nc{\ot}{\otimes}
 \nc{\otb}{\bigotimes}
 \nc{\pwl}{P^+}
 \nc{\ftil}{\tilde{f}}
 \nc{\etil}{\tilde{e}}
 \nc{\veps}{\varepsilon}
 \nc{\vphi}{\varphi}
 \nc{\ol}{\overline}

 \nc{\CB}{\mathsf{B}}
 \nc{\Bla}{\CB(\lambda)}
 \nc{\Binf}{\CB(\infty)}
 \nc{\CT}{\mathsf{T}}
 \nc{\Tla}{\CT(\lambda)}
 \nc{\Tinf}{\CT(\infty)}

 \nc{\BaCr}{\mathsf{C}}
 \nc{\bce}{c}
 \nc{\BaCrCirc}{\mathsf{C}^\circ}
 \nc{\BaCrBar}{\bar{\mathsf{C}}}
 \nc{\CupCirc}{\mathsf{C}_\sqcup^\circ}
 \nc{\CupBar}{\bar{\mathsf{C}}_\sqcup}

 \nc{\KaEm}{\mathsf{K}}
 \nc{\Kinf}{\mathsf{K}(\infty)}

 \DeclareMathOperator{\wt}{wt}

 \nc{\bxr}[2]{\boxed{#1}_{\;#2}}
 \nc{\nbxr}[2]{\#\,\boxed{#1}_{\;#2}}

 \nc{\relCE}[1]{\mathsf{E}{[#1]}}
 \nc{\relCA}[1]{\mathsf{A}{[#1]}}
 \nc{\relCB}[1]{\mathsf{B}{[#1]}}
 \nc{\relCC}[1]{\mathsf{C}{[#1]}}
 \nc{\relCD}[1]{\mathsf{D}{[#1]}}
 \nc{\relBA}[1]{\mathsf{A}{[#1]}}
 \nc{\relBB}[1]{\mathsf{B}{[#1]}}
 \nc{\relBC}[1]{\mathsf{C}{[#1]}}
 \nc{\relBD}[1]{\mathsf{D}{[#1]}}
 \nc{\relDA}[1]{\mathsf{A}{[#1]}}
 \nc{\relDB}[1]{\mathsf{B}{[#1]}}
 \nc{\relDC}[1]{\mathsf{C}{[#1]}}
 \nc{\relDD}[1]{\mathsf{D}{[#1]}}
 \nc{\relDE}[1]{\mathsf{E}{[#1]}}
 \nc{\relDF}[1]{\mathsf{F}{[#1]}}
 \nc{\relDH}[1]{\mathsf{E}{[#1]}}
 \nc{\relDI}[1]{\mathsf{F}{[#1]}}

 \nc{\comesfrom}{\leftsquigarrow}
 \nc{\boxcnt}[2]{\makebox[2.2em][l]{$t_{#1,#2}$}}

\begin{document}

\title{%
Description of~$\mathsf{B}(\infty)$ through Kashiwara Embedding\\ for $E_6$ and $E_7$ Lie Algebra Types} 

\author[1]{Jin Hong%
}
\author[2]{Hyeonmi Lee%
}
\affil[1]{\normalsize%
  Department of Mathematical Sciences and Research Institute of Mathematics\authorcr
  Seoul National University, Seoul 08826, Korea}
\affil[2]{\normalsize%
  Department of Mathematics and Research Institute for Natural Sciences\authorcr
  Hanyang University, Seoul 133, Korea}

\pagestyle{myheadings}
\date{}

\maketitle

\begin{abstract}
We study the crystal base $\Binf$ associated with the negative part of the quantum group for finite simple Lie algebras of types $E_6$ and~$E_7$.
We present an explicit description of~$\Binf$ as the image of a Kashiwara embedding that is in natural correspondence with the tableau description of~$\Binf$.

\end{abstract}


\section{Introduction}\label{sec:introduction}

The quantum group $U_q(\mathfrak{g})$ is a $q$-deformation of the universal enveloping algebra~$U(\mathfrak{g})$ for a Lie algebra~$\mathfrak{g}$ and the crystal base~$\Binf$ presents the bare skeleton structure of its negative part $U_q^-(\mathfrak{g})$.
The crystal~$\Binf$ has received attention since the very birth of crystal base theory~\cite{Ka91,Luszti90} as an integral part of the grand loop argument proving the existence of crystal bases and substantial efforts have been made to give explicit descriptions of the crystal~$\Binf$.
A variety of tools were used for this purpose and in particular, the works~\cite{Cliff98,Ho2005,Ka93,NaZe} use Kashiwara embedding.
This work presents new explicit descriptions of~$\Binf$ through the Kashiwara embedding for the $E_6$ and $E_7$ types that are in natural correspondence with the tableau description of~$\Binf$.

When the irreducible highest weight crystals~$\Bla$ corresponding to certain dominant integral weights~$\la$ are gathered together and some careful identifications are made among elements belonging to different crystals, one can give the resulting set of equivalence classes $\bigsqcup_\la \Bla/\negthickspace\sim$ a crystal structure induced from those of the~$\Bla$'s.
The resulting structure is known to be isomorphic to~$\Binf$ as a crystal, and this description of~$\Binf$ is valid for all symmetrizable Kac-Moody algebras~\cite{Ka02}.

The understanding of~$\Binf$ as~$\bigsqcup_\la \Bla/\negthickspace\sim$ can be used to create a more concrete realization of~$\Binf$.
The work~\cite{HonLee08} gathered together the Young tableau realizations~\cite{KasNak94,KanMis94} of the highest weight crystals~$\Bla$, available for the classical and $G_2$ finite simple Lie algebra types, and then made the appropriate identifications of crystal elements.
This resulted in a concrete realization of~$\Binf$, when the marginally large tableaux were chosen to represent the equivalence classes.
The Young tableau results analogous to~\cite{KasNak94,KanMis94} do not yet exist for the $E_6$, $E_7$, $E_8$, and~$F_4$ types, but the marginally large tableau realizations of~$\Binf$ for these types could still be obtained by~\cite{HonLee12} after just partially constructing the analogues.

There exists a strict crystal embedding~\cite{Ka93}
\begin{equation}
\Binf \hookrightarrow
\Binf\ot\CB_{i_m}\ot\cdots\ot\CB_{i_2}\ot\CB_{i_1},
\end{equation}
usually referred to as the Kashiwara embedding, for any sequence $\iota=(i_m,\dots,i_2,i_1)$ of elements from~$I$, the set of simple root indices.
Here, the crystals $\CB_i=\{\,b_i(k)\mid k\in\mathbf{Z}\,\}$, appearing on the right-hand side, are defined for each $i\in I$ and are extremely simple not only as sets, but also in terms of the Kashiwara operator actions.
For each of the classical finite simple Lie algebra types, the work~\cite{Cliff98} fixed an explicit choice of sequence~$\iota$ and described the image set of the corresponding Kashiwara embedding.
Because the image consists only of elements of the form
\begin{equation}
b_\infty\ot b_{i_m}(k_m)\ot\cdots\ot b_{i_1}(k_1),
\end{equation}
where $b_\infty$ is the highest weight element of~$\Binf$, this can be accepted as a simple realization of~$\Binf$.
The essence of describing the image set is in expressing the possible range of integers $(k_j)_{j=1}^m$, and the notion of large tableaux for the classical types, which the work introduced, was an important tool in this description.

The current paper provides results analogous to the descriptions of~$\Binf$ given by~\cite{Cliff98} for the $E_6$ and $E_7$ types.
The choice of sequences $\iota$ we use for the two types are such that the image set of the Kashiwara embedding is in natural correspondence with the tableau description of~$\Binf$ given by~\cite{HonLee12}.

Let us mention two existing results that give explicit descriptions of~$\Binf$, for the exceptional Lie algebra types.
The first is the work~\cite{Li1998} giving a description of $\Binf$ based on the generalized Gelfand-Tsetlin patterns.
The polyhedral realization~\cite{Ho2005} also gives a description of~$\Binf$. Their realizations for the~$E_6$ and~$E_7$ types correspond to sequences $\iota$ that are different from the ones used in this work.

The rest of this paper is organized as follows. 
In the next section, we recall some previous results and fix the notation to be used in the rest of the paper.
In Section~\ref{sec:KE.E6}, we develop our results for the $E_6$ type, and the $E_7$ type is discussed in Section~\ref{sec:KE.E7}.
The basic crystals that are used as box entries in our tableaux are explicitly given in the appendices.

\bigskip
\noindent\textbf{Acknowledgments}\quad
Both authors would like to thank the Department of Mathematics at U.C. Davis, where part of this work was performed, for its hospitality.
J.~Hong was supported by the Basic Science Research Program through the National Research Foundation of Korea (NRF) funded by the Ministry of Science, ICT \& Future Planning (NRF-2012R1A1B4003379).
H.~Lee was supported by the Basic Science Research Program through the National Research Foundation of Korea (NRF) funded by the Ministry of Education, Science and Technology (NRF-2012R1A1A2008392).

\section{Preliminaries}\label{sec:not}

In this section, we will recall some basic crystal base theory and the marginally large tableau description of crystal~$\Binf$.
We will not recall all the standard notation used with crystal base theory, but any non-standard ones to be used in this paper will be explained here.

The indexing schemes used in this paper for the $E_6$ and $E_7$ type root systems are as follows.
\begin{center}
\raisebox{11.5pt}{$E_6$:}\quad
\begin{tikzpicture}[scale=0.78, thick]
\draw (-2, 0) circle (2pt);
\draw (-1, 0) circle (2pt);
\draw ( 0, 0) circle (2pt);
\draw ( 1, 0) circle (2pt);
\draw ( 2, 0) circle (2pt);
\draw ( 0, 0.95) circle (2pt);
\draw (-1.93, 0) -- (-1.07, 0);
\draw (-0.93, 0) -- (-0.07, 0);
\draw ( 0, 0.07) -- ( 0, 0.88);
\draw ( 0.07, 0) -- ( 0.93, 0);
\draw ( 1.07, 0) -- ( 1.93, 0);
\draw (-1.77,-0.32) node {${}_1$};
\draw (-0.77,-0.32) node {${}_2$};
\draw ( 0.23,-0.32) node {${}_3$};
\draw ( 1.23,-0.32) node {${}_4$};
\draw ( 2.23,-0.32) node {${}_5$};
\draw ( 0.23, 0.65) node {${}_6$};
\end{tikzpicture}
\hspace{30pt}
\raisebox{11.5pt}{$E_7$:}\quad
\begin{tikzpicture}[scale=0.78, thick]
\draw (-3, 0) circle (2pt);
\draw (-2, 0) circle (2pt);
\draw (-1, 0) circle (2pt);
\draw ( 0, 0) circle (2pt);
\draw ( 1, 0) circle (2pt);
\draw ( 2, 0) circle (2pt);
\draw ( 0, 0.95) circle (2pt);
\draw (-2.93, 0) -- (-2.07, 0);
\draw (-1.93, 0) -- (-1.07, 0);
\draw (-0.93, 0) -- (-0.07, 0);
\draw ( 0, 0.07) -- ( 0, 0.88);
\draw ( 0.07, 0) -- ( 0.93, 0);
\draw ( 1.07, 0) -- ( 1.93, 0);
\draw (-2.77,-0.32) node {${}_7$};
\draw (-1.77,-0.32) node {${}_1$};
\draw (-0.77,-0.32) node {${}_2$};
\draw ( 0.23,-0.32) node {${}_3$};
\draw ( 1.23,-0.32) node {${}_4$};
\draw ( 2.23,-0.32) node {${}_5$};
\draw ( 0.23, 0.65) node {${}_6$};
\end{tikzpicture}
\end{center}
Our later arguments will rely on the following version of the tensor product rule that is suitable for applications to the tensor product of more than two crystals.
\begin{proposition}[\hspace{1sp}\cite{Ka93}]\label{prop:mtpr}
Let $B_k$ $(1\leq k\leq m)$ be crystals with $b_k\in B_k$, and let us set
\begin{equation*}
a_k = \veps_i(b_k)-\sum_{1\leq v<k} \langle h_i,\wt(b_v)\rangle.
\end{equation*}
Then we have the following.
\begin{enumerate}
\item
 $\etil_i(b_1\ot\cdots\ot b_m)
 = b_1\ot\cdots\ot b_{k-1}\ot\etil_i b_k\ot b_{k+1}\ot\cdots\ot b_m$,
 when $a_k>a_v$ for $1\leq v<k$ and $a_k\geq a_v$ for $k<v\leq m$.
\item
 $\ftil_i(b_1\ot\cdots\ot b_m)
 = b_1\ot\cdots\ot b_{k-1}\ot\ftil_i b_k\ot b_{k+1}\ot\cdots\ot b_m$,
 when $a_k\ge a_v$ for $1\leq v<k$ and $a_k>a_v$ for $k<v\leq m$.
\end{enumerate}
\end{proposition}

The abstract crystal $\CB_i = \{b_i(k) \mid k\in\mathbf{Z}\}$ was introduced in~\cite{Ka93} for all Kac-Moody Lie algebra types and each simple root index $i\in I$.
Its crystal structure is as follows.
\begin{alignat*}{3}
\wt(b_i(k))     &=k\alpha_i,\\
\vphi_i(b_i(k)) &=k,         & \veps_i(b_i(k))&=-k,\\
\vphi_i(b_j(k)) &=-\infty,   & \veps_i(b_j(k))&=-\infty,
& \quad&\text{for $j\neq i$},\\
\ftil_i(b_i(k))    &=b_i(k-1),  \quad & \etil_i(b_i(k))&=b_i(k+1),\\
\ftil_i(b_j(k))    &=0,         & \etil_i(b_j(k))&=0, & &\text{for $j\neq i$}.
\end{alignat*}

Kashiwara has shown~\cite{Ka93} the existence of an injective strict crystal morphism $\Binf \hookrightarrow \Binf\ot\CB_i$ that is uniquely determined by setting $b_\infty \mapsto b_\infty \ot b_i(0)$, where $b_\infty$ is the highest weight element of~$\Binf$.
This implies that there exists an injective strict crystal morphism
\begin{equation}
\Binf \hookrightarrow \Binf\ot\CB_{i_m}\ot\cdots\ot\CB_{i_2}\ot\CB_{i_1},
\end{equation}
determined by $b_\infty \mapsto b_\infty\ot b_{i_m}(0)\ot\cdots\ot b_{i_1}(0)$, for any sequence $\iota = (i_k)_{k=1}^m$ of simple root indices.
The goal of this work is to give explicit descriptions of~$\Binf$ for the types $E_6$ and~$E_7$, with specific choices of the sequence~$\iota$ for each type.

Let us next recall the theory of marginally large tableaux for the exceptional types from~\cite{HonLee12}.
The boxes of a (marginally) large tableau are filled with elements from the basic crystal, which we will denote by~$\BaCr$.
We have $\BaCr = \CB(\La_1)$ for the $E_6$ type and $\BaCr = \CB(\La_7)$ for the $E_7$ type.
The two crystal graphs are explicitly given in the appendix.
The directed graph~$\BaCr$ is acyclic and admits a partial order structure.
For two elements $\bce_1, \bce_2\in\BaCr$, we will write $\bce_2 \comesfrom \bce_1$ if either $\bce_1 = \bce_2$ or there exists a directed path from~$\bce_1$ to~$\bce_2$ on the directed graph~$\BaCr$.

The rows of a tableau will be numbered from bottom to top, so that the first row of a tableau is its bottom or shortest row.
There are separate restrictions on which elements of the basic crystal~$\BaCr$ may be placed in the boxes of a (marginally) large tableau for each of its rows.
For each row index~$r$, we take~$\BaCr_r$ to be the full sub-graph of the directed graph~$\BaCr$ whose nodes consist of all basic crystal elements that may appear on the $r$-th row of a (marginally) large tableau.
Specifically, the nodes for these directed graphs are as follows for the $E_6$-type tableaux.
\begin{equation}
\begin{aligned}
\BaCr_5 &= \BaCr = \CB(\La_1)\\
\BaCr_4 &= \{
 \bce\in\BaCr \mid  \bar{6}1 \comesfrom \bce \comesfrom \bar{1}2
\}\\
\BaCr_3 &= \{
 \bce\in\BaCr \mid  \bar{4}2 \comesfrom \bce \comesfrom \bar{2}3
\}\\
\BaCr_2 &= \{
 \bce\in\BaCr \mid  \bar{5}\bar{6}3 \comesfrom \bce \comesfrom \bar{3}46
\}\\
\BaCr_1 &= \{
 \bar{5}6, \bar{4}56
\}
\end{aligned}
\end{equation}
However, we emphasize that each $\BaCr_r$ is to be treated as a directed graph that inherits the arrows and also the partial order relation $\comesfrom$ from~$\BaCr$.
These directed graphs are fully illustrated in Appendix~\ref{app:E6BaCr}.

For the $E_7$-type tableaux, the basic crystal elements that are allowed in each row are as follows.
\begin{equation}
\begin{aligned}
\BaCr_6 &= \BaCr = \CB(\La_7)\\
\BaCr_5 &= \{
 \bce\in\BaCr \mid  \bar{5}7 \comesfrom \bce \comesfrom \bar{7}1
\}\\
\BaCr_4 &= \{
 \bce\in\BaCr \mid  \bar{6}1 \comesfrom \bce \comesfrom \bar{1}2
\}\\
\BaCr_3 &= \{
 \bce\in\BaCr \mid  \bar{4}2 \comesfrom \bce \comesfrom \bar{2}3
\}\\
\BaCr_2 &= \{
 \bce\in\BaCr \mid  \bar{5}\bar{6}3 \comesfrom \bce \comesfrom \bar{3}46
\}\\
\BaCr_1 &= \{
 \bar{5}6, \bar{4}56
\}
\end{aligned}
\end{equation}
The directed graphs~$C_5$ and~$C_6$ for the $E_7$ type are illustrated in Appendix~\ref{app:E7BaCr}.
Note that our indexing scheme for the $E_6$ and $E_7$ root systems and the labeling given to elements of the basic crystal~$\BaCr$ for the two types are such that the directed graphs $C_1$, $C_2$, $C_3$, and~$C_4$ are shared by the~$E_6$ and $E_7$ types and the directed graphs~$C_5$ for the two types are very similar.

One can easily observe that each directed graph~$\BaCr_r$ for both the $E_6$ and $E_7$ types has a unique source node and a unique sink node.
It is also easy to check that, for each node $\bce\in\BaCr_r$, every directed path from the source node of~$\BaCr_r$ to~$\bce$ has the same length and that the same can be stated of directed paths from~$\bce$ to the sink node of~$\BaCr_r$.
We will refer to these directed path lengths as the \emph{distance} from the source node to node~$\bce$ and the \emph{distance} from node~$\bce$ to the sink node.

A tableau~$T$ is \emph{large} (for the $E_6$ or $E_7$ types) if it satisfies the following conditions, for each possible row index~$r$.
\begin{enumerate}
\item Each box on the $r$-th row of $T$ is filled with a node from~$\BaCr_r$.
\item The number of boxes on the $r$-th row of~$T$ labeled by the source node of~$\BaCr_r$ is greater than the total number of boxes on its immediate lower row.
In particular, at least $r$ boxes on the $r$-th row of~$T$ contain the source node of~$\BaCr_r$.\label{tmpitemlab}
\item The basic crystal elements that appear on the $r$-th row of~$T$ can be placed on a directed path that joins the source node to the sink node of~$\BaCr_r$.
In other words, the set of elements appearing on a row of~$T$ is totally ordered with respect to the partial order~$\comesfrom$ on~$\BaCr_r$.
\item The elements on a row of~$T$, read from left to right, follow the order of these elements on the directed path that was just mentioned.
In other words, the positional order of the elements within each row respects the order~$\comesfrom$.
In particular, the $r$-th row left-end box of~$T$ is always filled with the source node of~$\BaCr_r$.
\end{enumerate}
A large tableau is \emph{marginally} large, if
\begin{enumerate}
\item[\ref{tmpitemlab}${}^\prime$] The number of boxes on the $r$-th row of~$T$ labeled by the source node of~$\BaCr_r$ is greater than the total number of boxes on its immediate lower row by exactly one.
\end{enumerate}
The two definitions given above are equivalent to those given by~\cite{HonLee12}.

The set of all marginally large tableaux is denoted by~$\Tinf$ and has a crystal structure.
The Kashiwara operators act on a marginally large tableau in mostly the same way as they would act on a normal tableau, i.e., through an \emph{admissible reading} of a tableau followed by an application of the tensor product rule and reassembly into a tableau.
However, sometimes a certain column, consisting only of source nodes, may have to be added or removed to make the resulting reassembled tableau marginally large.
We refer the reader to~\cite{HonLee12} for the full crystal structure.
This paper relies heavily on the following results.

\begin{proposition}[\hspace{1sp}\cite{HonLee12}]\label{prop:Tinf}
The crystal~$\Tinf$ is isomorphic to the crystal~$\Binf$.
\end{proposition}

It will be convenient to have the symbol~$\BaCrCirc_r$ denote a set that lacks just the single source node from~$\BaCr_r$.
The sets are
\begin{equation}
\begin{aligned}
&\BaCrCirc_5 =  \BaCr_5 \setminus \{1\},\quad
 \BaCrCirc_4 =  \BaCr_4 \setminus \{\bar{1}2\},\quad
 \BaCrCirc_3 =  \BaCr_3 \setminus \{\bar{2}3\},\\
&\BaCrCirc_2 =  \BaCr_2 \setminus \{\bar{3}46\},\quad\text{and}\quad
 \BaCrCirc_1 =  \BaCr_1 \setminus \{\bar{4}56\},
\end{aligned}
\end{equation}
for the $E_6$ type and
\begin{equation}
\begin{aligned}
&\BaCrCirc_6 =  \BaCr_6 \setminus \{7\},\quad
 \BaCrCirc_5 =  \BaCr_5 \setminus \{\bar{7}1\},\quad
 \BaCrCirc_4 =  \BaCr_4 \setminus \{\bar{1}2\},\quad
 \BaCrCirc_3 =  \BaCr_3 \setminus \{\bar{2}3\},\\
&\BaCrCirc_2 =  \BaCr_2 \setminus \{\bar{3}46\},\quad\text{and}\quad
 \BaCrCirc_1 =  \BaCr_1 \setminus \{\bar{4}56\},
\end{aligned}
\end{equation}
for the $E_7$ type.
As with~$\BaCr_r$, these should be treated as directed graphs that inherit all possible arrows and the partial order~$\comesfrom$ from~$\BaCr$.

We will require further notation for an even smaller sub-structures of~$\BaCr_r$ given by
\begin{equation}
\BaCrBar_r = \{\bce\in\BaCr_r \mid \text{$\bce$ has in-degree 1}\}.
\end{equation}
We caution the reader that, for each $\BaCr_r \neq \BaCr$, the source node of~$\BaCr_r$ is of in-degree 1 in~$\BaCr$, but does not belong to $\BaCrBar_r$, because it is of in-degree 0 in $\BaCr_r$.
The nodes for~$\BaCrBar_r$ form a subset of the nodes for~$\BaCrCirc_r$.
The directed graph structure of~$\BaCrBar_r$ will not be important to us, but we will still need its partial order~$\comesfrom$, inherited from that of~$\BaCr$ or~$\BaCr_r$.
For the $E_6$-type, the nodes of~$\BaCrBar_r$ may be listed explicitly as
\begin{equation}
\begin{aligned}
\BaCrBar_5 &= \left\{
 \begin{aligned}
 &\bar{5},   \bar{4}5,  \bar{3}4,  \bar{2}6,  \bar{6}1,
  \bar{3}16, \bar{4}2,  \bar{1}5,  \bar{2}15, \bar{3}25,
  \bar{5}6,  \bar{4}56,\\
 &\bar{6}4,  \bar{3}46, \bar{2}3,  \bar{1}2
 \end{aligned}
\right\},\\
\BaCrBar_4 &= \{
  \bar{6}1,  \bar{3}16,  \bar{4}2,  \bar{2}15, \bar{3}25,
  \bar{5}6,  \bar{4}56,  \bar{6}4,  \bar{3}46, \bar{2}3
\},\\
\BaCrBar_3 &= \{
  \bar{4}2,  \bar{3}25, \bar{5}6,  \bar{4}56, \bar{6}4,
  \bar{3}46
\},\\
\BaCrBar_2 &= \{
  \bar{5}6,  \bar{4}56, \bar{6}4
\},\\
\BaCrBar_1 &= \{
  \bar{5}6
\}.
\end{aligned}
\end{equation}
For the $E_7$-type, we have
\begin{equation}
\begin{aligned}
\BaCrBar_6 &= \left\{
 \begin{aligned}
 &\bar{7},   \bar{1}7,  \bar{2}1,  \bar{3}2,  \bar{4}6,
  \bar{6}5,  \bar{3}56, \bar{2}4,  \bar{5}7,  \bar{4}57,
  \bar{1}6,  \bar{3}47, \bar{2}67, \bar{6}1,\\
 &\bar{3}16, \bar{7}5,  \bar{1}57, \bar{4}2,  \bar{2}15,
  \bar{3}25, \bar{5}6,  \bar{6}4,  \bar{4}56, \bar{3}46,
  \bar{2}3,  \bar{1}2,  \bar{7}1
 \end{aligned}
\right\},\\
\BaCrBar_5 &= \left\{
 \begin{aligned}
 &\bar{5}7,  \bar{4}57, \bar{3}47, \bar{2}67, \bar{6}1,
  \bar{3}16, \bar{4}2,  \bar{1}57, \bar{2}15, \bar{3}25,
  \bar{5}6,  \bar{4}56,\\
 &\bar{6}4,  \bar{3}46, \bar{2}3,  \bar{1}2
 \end{aligned}
\right\},
\end{aligned}
\end{equation}
and the rest of the sets, $\BaCrBar_4$, $\BaCrBar_3$, $\BaCrBar_2$, and $\BaCrBar_1$, are identical to their corresponding $E_6$-type sets.
These are the circled nodes appearing in the directed graphs given in the appendix.

Given a (marginally) large tableau~$T$, for each row index~$r$ and $\bce\in\BaCrCirc_r$, we define
\begin{equation}\label{eq:tib}
t_{r,\bce} = \left(\,
\begin{minipage}{22em}
the number of boxes appearing on the $r$-th row of~$T$ containing $x\in\BaCr_r$ such that $x \comesfrom \bce$
\end{minipage}\,
\right).
\end{equation}
This definition appeared previously in~\cite{HonLee14} for the classical Lie algebra types, with the small extension that set $t_{r,\bce}$ to infinity for the source node~$\bce$ of~$\BaCr_r$.
If $\bce\in\BaCr_r$ appears in one of the boxes on the $r$-th row of~$T$, then $t_{r,\bce}$ is the number of all boxes containing $\bce$ and all the boxes appearing to their right.
We take the disjoint union
\begin{equation}\label{eq:CupCirc}
\CupCirc = \bigsqcup_{r = \text{all rows}} \BaCrCirc_r
\end{equation}
and set
\begin{equation}
t_{\bce} = t_{r,\bce},
\end{equation}
for each $\bce\in\BaCrCirc_r \subset \CupCirc$.
The symbols $t_{r,\bce}$ and $t_{\bce}$ will be used interchangeably.
The collection of non-negative integers
\begin{equation}\label{eq:abc}
(t_{\bce})_{\bce\in\CupCirc}
=
(t_{r,\bce})_{r = \text{all rows},\; \bce\in\BaCrCirc_r}
\end{equation}
will be referred to as the set of \emph{accumulated box counts} for a tableau.

For a row index~$r$ and $\bce\in\BaCr_r$, we will use $\nbxr{\bce}{r}$ to denote the number of times $\bce$ appear in the boxes on the $r$-th row of a tableau.

\section{Description of $\Binf$ through Kashiwara Embedding for Type-$E_6$}\label{sec:KE.E6}

Before working with the Kashiwara embedding, we will study some properties of~$\Tinf$, the marginally large tableau realization of~$\Binf$.
It is rather easy to reconstruct a marginally large tableau~$T$ from its full set of accumulated box counts $(t_{\bce})_{\bce\in\CupCirc}$.
For example, referring to the directed graph~$\BaCr_3$, as given by Appendix~\ref{app:E6BaCr}, one can devise the following steps that recursively reveals the number of boxes containing each crystal element that should appear on the third row of~$T$.
\begin{equation}\label{eq:revexample}
\begin{aligned}
\nbxr{\bar{4}2}{3} &= t_{3,\bar{4}2}\\
\nbxr{\bar{3}\bar{5}24}{3} &= t_{3,\bar{3}\bar{5}24} - t_{3,\bar{4}2}\\
\nbxr{\bar{5}\bar{6}3}{3} &= t_{3,\bar{5}\bar{6}3} - t_{3,\bar{3}\bar{5}24}\\
\nbxr{\bar{3}25}{3} &= t_{3,\bar{3}25} - t_{3,\bar{3}\bar{5}24}\\
\nbxr{\bar{5}6}{3} &= t_{3,\bar{5}6} - t_{3,\bar{5}\bar{6}3}\\
\nbxr{\bar{4}\bar{6}35}{3} &= t_{3,\bar{4}\bar{6}35} - t_{3,\bar{3}25} - \nbxr{\bar{5}\bar{6}3}{3}\\
\nbxr{\bar{4}56}{3} &= t_{3,\bar{4}56} - t_{3,\bar{4}\bar{6}35} - \nbxr{\bar{5}6}{3}\\
\nbxr{\bar{6}4}{3} &= t_{3,\bar{6}4} - t_{3,\bar{4}\bar{6}35}\\
\nbxr{\bar{3}46}{3} &= t_{3,\bar{3}46} - t_{3,\bar{4}56} - \nbxr{\bar{6}4}{3}
\end{aligned}
\end{equation}
The missing count $\nbxr{\bar{2}3}{3}$ for the leftmost box of the third row can be recovered from the condition for the tableau to be \emph{marginally} large, once the tableau's first and second rows are ready.

This shows that a full set of accumulated box counts \emph{uniquely} determines the original marginally large tableau, but the next lemma implies that not all of the accumulated box counts are required.

\begin{lemma}\label{lem:tmin}
Let $(t_{\bce})_{\bce\in\CupCirc}$ be the set of accumulated box counts for a \textup{(}marginally\textup{)} large tableau.
If $\bce_0$, $\bce_1$, and~$\bce_2$ \textup{(}$\bce_1\neq\bce_2$\textup{)} are nodes of the directed graph~$\BaCrCirc_r \subset \CupCirc$ that are connected by \textup{(}single-hop\textup{)} arrows $\bce_0 \leftarrow \bce_1$ and $\bce_0 \leftarrow \bce_2$, then $t_{\bce_0} = \min(t_{\bce_1},t_{\bce_2})$.
\end{lemma}
\begin{proof}
Let us consider the sets $G_k = \{x\in\BaCr_r \mid x \comesfrom \bce_k\}$ ($k=0,1,2$).
We can check from the directed graphs of Appendix~\ref{app:E6BaCr} that the partial order structure on~$\BaCr_r$ is such that $G_1\cap G_2 = G_0$.
In particular, we have $G_1\setminus G_2 = G_1\setminus G_0$ and $G_2\setminus G_1 = G_2\setminus G_0$.

The assumption $x_1 \comesfrom x_2$ for a pair of elements $x_1\in G_1\setminus G_2$ and $x_2 \in G_2\setminus G_1$ implies $x_1\comesfrom x_2\comesfrom \bce_2$ and the contradiction $x_1\in G_2$, so that no pair of elements from $G_1\setminus G_0$ and $G_2\setminus G_0$ can be on the same directed path of~$\BaCr_r$.
In other words, the $r$-th row of a large tableau~$T$ cannot contain elements from both $G_1\setminus G_0$ and $G_2\setminus G_0$ at the same time.

Since $G_0 \subset G_1$ and $G_0 \subset G_2$, if the $r$-th row of~$T$ contains no element of $G_1\setminus G_0$, then we must have $t_{\bce_1} = t_{\bce_0} \leq t_{\bce_2}$.
We must similarly have $t_{\bce_2} = t_{\bce_0} \leq t_{\bce_1}$, when $T$ contains no elements of $G_2\setminus G_0$.
In both cases, we have $t_{\bce_0} = \min(t_{\bce_1},t_{\bce_2})$.
\end{proof}

Let us collect the properties we know of the set of accumulated box counts into a definition.

\begin{definition}\label{def:pathcons1}
For each row index~$r$, let $\BaCr^\prime_r$ be a full sub-graph of the directed graph~$\BaCrCirc_r$.
Each $\BaCr^\prime_r$ inherits the partial order $\comesfrom$ from~$\BaCr_r$.
A set of non-negative integers $(s_{\bce})_{\bce\in\bigsqcup_r\BaCr^\prime_r}$ is \emph{path-consistent}, if we have
\begin{enumerate}
\item $s_{\bce_1} \leq s_{\bce_2}$ for every $\bce_1, \bce_2\in\BaCr^\prime_r \subset \bigsqcup_r\BaCr^\prime_r$ such that $\bce_1 \comesfrom \bce_2$ and
\item $s_{\bce} = \min(s_{\bce_1},s_{\bce_2})$ for every $\bce, \bce_1, \bce_2 \in\BaCr^\prime_r \subset \bigsqcup_r\BaCr^\prime_r$ ($\bce_1\neq\bce_2$) that are connected by (single-hop) arrows $\bce \leftarrow \bce_1$ and $\bce \leftarrow \bce_2$, .
\end{enumerate}
\end{definition}

It is clear from Lemma~\ref{lem:tmin} that any set of accumulated box counts for a large tableau is path-consistent, and we will soon show that it also works the other way around.
The term path-consistent is a reflection of the condition which requires the set of box labels appearing in the $r$-th row of a large tableau to reside on a directed path of~$\BaCr_r$.

\begin{lemma}\label{lem:capmin}
Let $\mathbf{s} = (s_{\bce})_{\bce\in\CupCirc}$ be a path-consistent set of non-negative integers, and let $\bce_1, \bce_2 \in \BaCrCirc_r$ be two nodes that cannot be placed on a directed path of~$\BaCr_r$, i.e., they satisfy neither $\bce_1 \comesfrom \bce_2$ nor $\bce_2 \comesfrom \bce_1$.
Set $G_1 = \{x\in\BaCrCirc_r \mid x\comesfrom \bce_1\}$ and $G_2 = \{x\in\BaCrCirc_r \mid x\comesfrom \bce_2\}$.
Then, we have $G_1\cap G_2 = \{x\in\BaCrCirc_r \mid x\comesfrom \bce_0\}$ for some node $\bce_0\in\BaCrCirc_r$.
Furthermore, for this~$\bce_0$, we have $s_{\bce_0} = \min(s_{\bce_1},s_{\bce_2})$.
\end{lemma}
\begin{proof}
Referencing the directed graphs $\BaCrCirc_r$ from Appendix~\ref{app:E6BaCr}, one can easily check that it is always possible to locate a \emph{rectangle} of the form
\begin{center}
\includegraphics{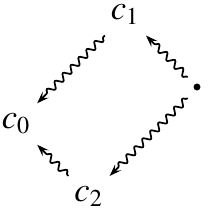}
\end{center}
from within~$\BaCrCirc_r$, for any pair of $\bce_1$ and $\bce_2$ having the assumed properties.
This observation and the layout of the graphs make the first claim evident.
It only remains to show that $s_{\bce_0} = \min(s_{\bce_1},s_{\bce_2})$ holds for the $\bce_0$ found at the corner of the rectangle.

Let us consider the following directed graph of rectangle shape and assume that $(s_{k})_{k=1}^{12}$ is a set of non-negative integers, labeled by their twelve nodes, that satisfies the two conditions of Definition~\ref{def:pathcons1}, with the indices suitably adjusted to fit the current situation.
\begin{center}
\includegraphics{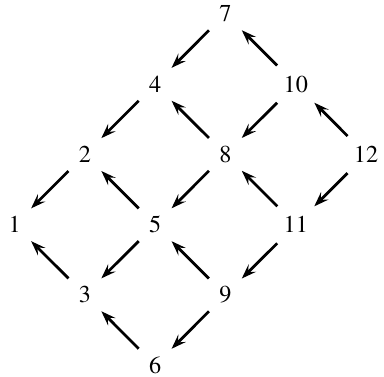}
\end{center}
Then, it is easy to check that
\begin{equation*}
\begin{aligned}
s_{1}
 &= \min(s_2,s_3)\\
 &= \min\big(\min(s_4,s_5),\min(s_5,s_6)\big)
  = \min(s_4,s_5,s_6)\\
 &= \min\big(\min(s_7,s_8),\min(s_8,s_9),s_6\big)
  = \min(s_7,s_8,s_9,s_6)\\
 &= \min\big(s_7,\min(s_{10},s_{11}),s_9,s_6\big)
  = \min(s_7,s_{10},s_{11},s_9,s_6)\\
 &= \min(s_7,s_6).
\end{aligned}
\end{equation*}
Here, the final equality follows from the inequalities $s_7\leq s_{10}$ and $s_6 \leq s_9 \leq s_{11}$ implied by the first condition of Definition~\ref{def:pathcons1}.
The two nodes on the final line, namely, $6$ and~$7$, are two vertices of the rectangle that are facing each other.

Note that the rectangle that we previously located within~$\BaCr_r$, which takes $\bce_1$, $\bce_2$, and~$\bce_0$ as three of its vertices, has none of its \emph{internal} arrows missing.
The above argument is general enough to be applicable to this rectangle.
\end{proof}

The conclusions of this lemma apply to any full set of accumulated box counts for a large tableau, since it is path-consistent by Lemma~\ref{lem:tmin}.
We can now show that every path-consistent full set of non-negative integers appears as the set of accumulated box counts for some tableau.

\begin{lemma}\label{lem:ABCsurj}
The set of accumulated box counts $(t_{\bce})_{\bce\in\CupCirc}$ for a \textup{(}marginally\textup{)} large tableau~$T$ is path-consistent.
Conversely, given any path-consistent set of non-negative integers $\mathbf{s} = (s_{\bce})_{\bce\in\CupCirc}$, there exists a unique marginally large tableau whose set of accumulated box counts is~$\mathbf{s}$.
\end{lemma}
\begin{proof}
The first claim follows from the definition for a large tableau and Lemma~\ref{lem:tmin}, and we have already discussed the uniqueness appearing in the second claim.
It suffices to show that every path-consistent set of non-negative integers indexed by~$\CupCirc$ is the set of accumulated box counts for some marginally large tableau.

Let us first proceed as with the example~\eqref{eq:revexample} and generate a set of numbers $\mathbf{n} = (n_{\bce})_{\bce\in\CupCirc}$, which could be interpreted as the (non-accumulated) box counts for a large tableau.
For each row index~$r$, one starts from the sink node of the directed graph~$\BaCr_r$ by setting $n_{\text{sink node}} = s_{\text{sink node}}$ and works inductively.
Once $n_{\bce}$ has been computed for all $\bce\in\BaCrCirc_r$ that are at a certain distance~$k$ to the sink node, one treats the nodes that are at distance $k+1$ to the sink node.
More precisely, we set
\begin{equation*}
n_{\bce} = s_{\bce} - \sum_{x\comesfrom \bce, x\neq \bce} n_{x}
\end{equation*}
on every occasion for each $\bce\in\BaCrCirc_r$.
The order in which this operation is carried out among nodes that are at equal distance to the sink node is clearly irrelevant.

By construction, the \emph{accumulated counts} for the~$\mathbf{n}$ generated in this manner will be~$\mathbf{s}$.
That is, we will have
\begin{equation*}
s_{\bce} = \sum_{x\comesfrom \bce} n_{x},
\end{equation*}
for every $\bce\in\BaCrCirc_r$.
However, it is not yet clear if all the non-zero $n_{\bce}$'s can be placed on a directed path that connects the source node to the sink node of~$\BaCr_r$, or even if the integers $n_{\bce}$ are non-negative.
Let us show that these two properties hold through an induction on the distance to the sink node of~$\BaCrCirc_r$.

Our induction hypothesis is that the set of $n_{\bce}$'s with the node~$\bce$ restricted to those that are at most distance~$k$ to the sink node of~$\BaCrCirc_r$ are non-negative and that the collection of the nodes corresponding to the non-zero $n_{\bce}$'s satisfies the on-a-path requirement.

Let $\bce$ be at distance $k+1$ to the sink node of~$\BaCrCirc_r$.
If $\bce_1\leftarrow \bce$ is the only arrow leaving from~$\bce$, then we know from the way $n_{\bce}$ was computed that
\begin{equation*}
n_{\bce} = s_{\bce} - s_{\bce_1}.
\end{equation*}
If there are two arrows $\bce_1\leftarrow \bce$ and $\bce_2\leftarrow \bce$ leaving from~$\bce$, since our induction hypothesis ensures that the nodes $\bce_1$ and~$\bce_2$ reside in the region satisfying the on-a-path condition, we can follow the arguments given in the proof of Lemma~\ref{lem:tmin} to state
\begin{equation*}
n_{\bce} = s_{\bce} - \max(s_{\bce_1}, s_{\bce_2}).
\end{equation*}
Both of these right-hand side values must be non-negative integers by the first condition of Definition~\ref{def:pathcons1}, so that all integers $n_{\bce}$ at up to distance $k+1$ to the sink node of~$\BaCrCirc_r$ must be non-negative.

Now, suppose we have two positive integers $n_{\bce_1}$ and $n_{\bce_2}$ for some $\bce_1$ and $\bce_2$ that lie within distance $k+1$ to the sink of~$\BaCrCirc_r$, and let us assume that the two cannot be placed on a directed path of~$\BaCr_r$.
Setting $G_1 = \{x\in\BaCrCirc_r \mid x\comesfrom \bce_1\}$ and $G_2 = \{x\in\BaCrCirc_r \mid x\comesfrom \bce_2\}$, we know from Lemma~\ref{lem:capmin} that $G_1\cap G_2 = \{x\in\BaCrCirc_r \mid x\comesfrom \bce_0\}$, for some node $\bce_0\in\BaCrCirc_r$, and that $s_{\bce_0} = \min(s_{\bce_1},s_{\bce_2})$.
However, since both $n_{\bce_1}$ and $n_{\bce_2}$ were chosen to be positive, $s_{\bce_0} = \sum_{x\comesfrom \bce_0} n_{x}$ must be strictly smaller than both $s_{\bce_1}$ and $s_{\bce_2}$.
This contradiction completes our induction step.

Thus, we have obtained a set of non-negative integers $\mathbf{n} = (n_{\bce})_{\bce\in\CupCirc}$ having properties that allow it to be interpreted as the set of (non-accumulated) box counts for a large tableau.
As was with the example~\eqref{eq:revexample}, the missing counts for the leftmost boxes on each tableau row, i.e., the boxes labeled by the source node of~$\BaCr_r$, may be determined from the condition for the tableau to be \emph{marginally} large.
\end{proof}

The combination of Lemma~\ref{lem:tmin} and Lemma~\ref{lem:ABCsurj} results in a more compact description of a marginally large tableau in terms of its accumulated box counts.
We introduce another index set
\begin{equation}\label{eq:CupBar}
\CupBar = \bigsqcup_{r = \text{all rows}} \BaCrBar_r
\end{equation}
that is similar to~\eqref{eq:CupCirc}.

\begin{proposition}\label{prop:bijection}
The reduced set of accumulated box counts $\bar{\mathbf{t}} = (t_{\bce})_{\bce\in\CupBar}$ for a \textup{(}marginally\textup{)} large tableau is path-consistent.
Conversely, given any path-consistent set of non-negative integers $\bar{\mathbf{s}} = (s_{\bce})_{\bce\in\CupBar}$, there exists a unique marginally large tableau whose reduced set of accumulated box counts is~$\bar{\mathbf{s}}$.
\end{proposition}
\begin{proof}
The first claim follows trivially from the first claim of Lemma~\ref{lem:ABCsurj}, since $\CupBar \subset \CupCirc$.
To prove the second claim, it suffices to show that $\bar{\mathbf{s}}$ can be extended uniquely to a set of non-negative integers $\mathbf{s} = (s_{\bce})_{\bce\in\CupCirc}$ that is also path-consistent, as the extended $\mathbf{s}$ may be connected to a unique marginally large tableau through Lemma~\ref{lem:ABCsurj}.

Since every node of $\CupCirc \setminus \CupBar$ has in-degree 2, there can be at most one approach to the extension, if the second condition of Definition~\ref{def:pathcons1} is to be satisfied by the result.
For each row index~$r$, one starts from the node(s) at distance $1$ from the source node of~$\BaCr_r$ and works inductively on the distance from the source node to obtain a full set of integers~$\mathbf{s}$ indexed by~$\CupCirc$.
Note that all the nodes at distance $1$ from the source node belong to~$\CupBar$, so that the base case is ready.
When the~$s_{\bce}$ values for every~$\bce$ at distance~$k$ from the source node are ready, one adds in every missing $s_{\bce}$ value through the $s_{\bce} = \min(s_{\bce_1},s_{\bce_2})$ rule for $\bce$ at distance $k+1$ from the source node.
It is clear that the fully extended~$\mathbf{s}$ consists of non-negative integers, but it remains to check whether~$\mathbf{s}$ satisfies the two conditions of Definition~\ref{def:pathcons1}.

Dealing with the second condition of Definition~\ref{def:pathcons1} is easier.
The addition of each $s_{\bce}$ is made to automatically satisfy the second condition in relation to the parent nodes of~$\bce$.
In the other direction, at the time when $s_{\bce}$ is being added, every node~$\bce_1$ that is further away from the source node than~$\bce$ with a $s_{\bce_1}$ value already assigned to it belongs to $\CupBar$ and has an in-degree of~$1$, so the second condition is vacuous in relation to these nodes.

Let us now deal with the first condition of Definition~\ref{def:pathcons1}.
We assume a path-consistent set of non-negative integers $(s_{\bce})_{\bce\in\BaCr^\prime_r}$ and suppose that we are adding one more $s_{\bce}$ value to this set through an application of the second condition of Definition~\ref{def:pathcons1}.
That is, we assume that there are arrows $\bce \leftarrow \bce_1$ and $\bce \leftarrow \bce_2$, for some $\bce\not\in\BaCr^\prime_r$ and $\bce_1, \bce_2\in \BaCr^\prime_r$, and set
\begin{equation*}
s_{\bce} = \min(s_{\bce_1},s_{\bce_2}).
\end{equation*}
Since the only in-bound arrows for~$\bce$ are from~$\bce_1$ and~$\bce_2$, if $\bce\comesfrom x$ for some $x\in\BaCr^\prime_r$, then we must have either $\bce_1 \comesfrom x$ or $\bce_2 \comesfrom x$.
This implies either $s_{\bce} \leq s_{\bce_1} \leq s_{x}$ or $s_{\bce} \leq s_{\bce_2} \leq s_{x}$, so that $s_{\bce} \leq s_{x}$.
In the opposite direction, if $x \comesfrom \bce$ for some $x\in\BaCr^\prime_r$, then we have both $x \comesfrom \bce_1$ and $x\comesfrom \bce_2$, so that $s_{x} \leq s_{\bce_1}$ and $s_{x} \leq s_{\bce_2}$.
Hence, we must have $s_{x} \leq \min(s_{\bce_1},s_{\bce_2}) = s_{\bce}$, and this completes the proof.
\end{proof}

Since $\BaCrBar_r$ consists of the in-degree $1$ nodes from~$\BaCr_r$, the second condition of Definition~\ref{def:pathcons1} is vacuous on~$\BaCrBar_r$.
In other words, the notation of being path-consistent is very simple for a set of non-negative integers indexed by~$\CupBar$.

\begin{remark}
A set of non-negative integers $(s_{\bce})_{\CupBar}$ is path-consistent if and only if we have $s_{\bce_1} \leq s_{\bce_2}$ for every $\bce_1,\bce_2\in\BaCrBar_r\subset\CupBar$ such that $\bce_1 \comesfrom \bce_2$.
\end{remark}

Recall from Proposition~\ref{prop:Tinf} that $\Tinf$, the set of all marginally large tableaux, is a realization of the crystal~$\Binf$.
Proposition~\ref{prop:bijection} provides an explicit bijection between $\Tinf$ and the set
\begin{equation}\label{eq:rawset}
\left\{
  (s_{\bce})_{\bce}
  {\textstyle \in \prod_{\bce\in\CupBar} \mathbf{Z}}
\;\left\vert\;
\begin{minipage}{16em}
(a)~$s_{\bce} \geq 0$ for every $\bce\in\CupBar$;\\
(b)~$s_{\bce_1} \leq s_{\bce_2}$ for every $\bce_1, \bce_2\in\BaCrBar_r \subset \CupBar$ such that $\bce_1 \comesfrom \bce_2$
\end{minipage}
\right.
\right\},
\end{equation}
which could potentially be easier to handle than~$\Tinf$.
Our next goal is to interpret the set~\eqref{eq:rawset} as a crystal so that the bijection becomes a crystal isomorphism.

Noting that the abstract crystal $\CB_i$, explained in Section~\ref{sec:not}, is identical $\mathbf{Z}$ as a set, we wish to essentially replace the $\prod_{\bce\in\CupBar} \mathbf{Z}$ of~\eqref{eq:rawset} with something that could be expressed as $\otb_{\bce\in\CupBar} \CB_\bce$, where each $\CB_\bce$ is set to one of the $\CB_i$'s.
However, since the tensor product rule for abstract crystals is not symmetric on the two components and the index set $\CupBar$ is not linearly ordered, such an expression is ambiguous.
Below, we will clarify the meaning of $\otb_{\bce\in\CupBar} \CB_\bce$ by first describing the object we wish to represent with this expression and then connecting the object with the tensor product indexed by~$\CupBar$.

Consider any directed path on the $E_6$-type directed graph~$\BaCr_5$ that connects its unique source node to its unique sink node.
This is a sequence of arrow colors $i\in I$, and it can readily be understood to be a tensor product of the corresponding~$\CB_i$'s.
More precisely, given a source-to-sink directed path
\begin{equation}
\mathbf{p}_5 = \big(
\bar{5}
\xleftarrow{\ i_{16}\ } \cdots \cdots \cdots
\xleftarrow{\ i_2\ } \;
\xleftarrow{\ i_1\ }
1 \big)
\end{equation}
on the $E_6$-type $\BaCr_5$, we consider the tensor product of abstract crystals
\begin{equation}
\KaEm_{\mathbf{p}_5}
 = \CB_{i_{16}} \ot \cdots \ot \CB_{i_2} \ot \CB_{i_1}.
\end{equation}
For example, given the source-to-sink directed path on~$\BaCr_5$
\begin{equation}
\mathbf{u} = \big(
\bar{5}
\xleftarrow{\,5}\;
\xleftarrow{\,4}\;
\xleftarrow{\,3}\;
\xleftarrow{\,2}\;
\xleftarrow{\,1}\;
\xleftarrow{\,6}\;
\xleftarrow{\,3}\;
\xleftarrow{\,2}\;
\xleftarrow{\,4}\;
\xleftarrow{\,3}\;
\xleftarrow{\,6}\;
\xleftarrow{\,5}\;
\xleftarrow{\,4}\;
\xleftarrow{\,3}\;
\xleftarrow{\,2}\;
\xleftarrow{\,1}
1\big),
\end{equation}
we set
\begin{equation}\label{eq:explicitBt3}
\KaEm_{\mathbf{u}}
=
\begin{aligned}
&\CB_5\ot\CB_4\ot\CB_3\ot\CB_2
 \ot\CB_1\ot\CB_6\ot\CB_3\ot\CB_2\\
&\ot\CB_4\ot\CB_3\ot\CB_6\ot\CB_5
 \ot\CB_4\ot\CB_3\ot\CB_2\ot\CB_1,
\end{aligned}
\end{equation}
and for the source-to-sink directed path
\begin{equation}
\mathbf{v} = \big(
\bar{5}
\xleftarrow{\,5}\;
\xleftarrow{\,4}\;
\xleftarrow{\,3}\;
\xleftarrow{\,6}\;
\xleftarrow{\,2}\;
\xleftarrow{\,3}\;
\xleftarrow{\,4}\;
\xleftarrow{\,5}\;
\xleftarrow{\,1}\;
\xleftarrow{\,2}\;
\xleftarrow{\,3}\;
\xleftarrow{\,4}\;
\xleftarrow{\,6}\;
\xleftarrow{\,3}\;
\xleftarrow{\,2}\;
\xleftarrow{\,1}
1\big),
\end{equation}
we take
\begin{equation}\label{eq:explicitBt4}
\KaEm_{\mathbf{v}}
=
\begin{aligned}
&\CB_5\ot\CB_4\ot\CB_3\ot\CB_6
 \ot\CB_2\ot\CB_3\ot\CB_4\ot\CB_5\\
&\ot\CB_1\ot\CB_2\ot\CB_3\ot\CB_4
 \ot\CB_6\ot\CB_3\ot\CB_2\ot\CB_1.
\end{aligned}
\end{equation}

Now, notice that for every commutative box
\begin{center}
\includegraphics{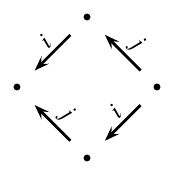}
\end{center}
that appears within~$\BaCr_5$, the nodes for the colors~$i$ and~$j$ are not connected in our Dynkin diagram for type~$E_6$, so that we have a trivial isomorphism $\CB_i\ot\CB_j \cong \CB_j\ot\CB_i$ of crystals.
In other words, the structure of $\KaEm_{\mathbf{p}_5}$ does not depend on~$\mathbf{p}_5$, as long as it is chosen to be a source-to-sink directed path on~$\BaCr_5$, and, furthermore, the isomorphisms between the different crystals $\BaCr_{\mathbf{p}_5}$ are trivial reordering of the components.
Based on this observation, we can set $\KaEm_{5}$ to the common crystal structure of the~$\KaEm_{\mathbf{p}_5}$'s that correspond to source-to-sink directed paths~$\mathbf{p}_5$ on~$\BaCr_5$.

The above argument can be repeated for any row index~$r$.
Each directed graph~$\BaCr_r$ has a unique source node and a unique sink node.
Each source-to-sink directed path~$\mathbf{p}_r$ on~$\BaCr_r$ defines a crystal~$\KaEm_{\mathbf{p}_r}$.
The crystals $\KaEm_{\mathbf{p}_r}$ defined using different $\mathbf{p}_r$ are trivially isomorphic to each other, and we let
\begin{equation}
\KaEm_{r} = \left(\,
\begin{minipage}{22em}
the common crystal structure of~$\KaEm_{\mathbf{p}_r}$ associated with source-to-sink directed paths~$\mathbf{p}_r$ on~$\BaCr_r$
\end{minipage}\,
\right).
\end{equation}
We can now state that the crystal structure we wish to represent with $\otb_{\bce\in\CupBar} \CB_\bce$ is
\begin{equation}
\KaEm =
  \KaEm_{5} \ot \KaEm_{4} \ot \KaEm_{3} \ot \KaEm_{2} \ot \KaEm_{1}.
\end{equation}

Let us next connect the crystal~$\KaEm$ to the expression $\otb_{\bce\in\CupBar} \CB_\bce$.
Referring to the $E_6$-type directed graph~$\BaCr_5$, as given by the first diagram of Appendix~\ref{app:E6BaCr}, one can form a natural grouping of the arrows based on their colors and arrangements.
For example, there are six $1$-arrows in~$\BaCr_5$, and five of these that are side by side to each other can be grouped together and we can let the single $1$-arrow at the right end form a separate group by itself.
We can similarly see the six $2$-arrows separated into three groups.
The two leftmost $2$-arrows that lie side by side form one group, the three $2$-arrows gathered near the center form another group, and the single $2$-arrow at the right forms a solo-group.
There are four $3$-arrow groups, three $4$-arrow groups, two $5$-arrow groups, and two $6$-arrow groups.
In all, we can partition the set of all arrows appearing in~$\BaCr_5$ into 16 separate groups of arrows.

Recall that the circled elements in the first diagram of Appendix~\ref{app:E6BaCr} are the nodes of~$\BaCrBar_5$.
For each group of arrows, there is precisely one element of~$\BaCrBar_5$ that is positioned at the head of one of the arrows.
For example, $\bar{1}5$ is the only circled node that receives an arrow belonging the aforementioned group of five $1$-arrows, and a circled $\bar{1}2$ is found at the head of the other $1$-arrow solo-group.
Similarly, $\bar{2}6$, $\bar{2}15$, and $\bar{2}3$ are the three elements of~$\BaCrBar_5$ corresponding to the three $2$-arrow groups of sizes two, three, and one, respectively.

We have thus explained that there is a natural partition of the arrows appearing in the directed graph~$\BaCr_5$ and that the groups of arrows are in one-to-one correspondence with elements of~$\BaCrBar_5$.
Now, it is clear that any source-to-sink directed path on~$\BaCr_5$ will pass through each and every group of arrows precisely once.
In fact, this property could have been used to \emph{define} the proper grouping of arrows in a compact manner.
Gathering what we have discussed so far, we can understand the expression $\otb_{\bce\in\BaCrBar_r} \CB_\bce$ as the common crystal structure~$\KaEm_r$ of the~$\KaEm_{\mathbf{p}_r}$'s and understand $\otb_{\bce\in\CupBar} \CB_\bce$ as expressing~$\KaEm$.

The notation $\otb_{\bce\in\CupBar} \CB_\bce$ is such that each $\CB_{\bce}$ always refers to the same $\CB_i$ ($i\in I$), and we can explicitly list them as
\begin{equation}\label{eq:ctoi}
\begin{alignedat}{6}
 \CB_{\bar{5}\phantom{2}} &= \CB_5,\quad
&\CB_{\bar{4}5\phantom{2}} &= \CB_4,\quad
&\CB_{\bar{3}4\phantom{1}} &= \CB_3,\quad
&\CB_{\bar{2}6\phantom{1}} &= \CB_2,\quad
&\CB_{\bar{6}1} &= \CB_6,\quad
&\CB_{\bar{3}16} &= \CB_3,\\
 \CB_{\bar{4}2} &= \CB_4,
&\CB_{\bar{1}5\phantom{2}} &= \CB_1,
&\CB_{\bar{2}15} &= \CB_2,
&\CB_{\bar{3}25} &= \CB_3,
&\CB_{\bar{5}6} &= \CB_5,
&\CB_{\bar{4}56} &= \CB_4,\\
 \CB_{\bar{6}4} &= \CB_6,
&\CB_{\bar{3}46} &= \CB_3,
&\CB_{\bar{2}3\phantom{1}} &= \CB_2,
&\CB_{\bar{1}2\phantom{6}} &= \CB_1,
\end{alignedat}
\end{equation}
for the $E_6$ type.
However, the expression $\otb_{\bce\in\CupBar} \CB_\bce$ is meant to convey more structure or positional order among the tensor components than is afforded by this simple correspondence.
For example, the component $\CB_{\bar{5}6}$ of $\otb_{\bce\in\BaCrBar_5} \CB_\bce$ refers to the 12-th component of~\eqref{eq:explicitBt3} and the 8-th component of~\eqref{eq:explicitBt4}, and these should not be confused with the first components of~\eqref{eq:explicitBt3} and~\eqref{eq:explicitBt4}, even though they also happen to be~$\CB_5$ ($5\in I$).

The whole purpose of developing the expression $\otb_{\bce\in\CupBar} \CB_\bce$, rather than forcing a trivial sequential indexing scheme onto the lowest level tensor product components of~$\KaEm_{\mathbf{p}_5}\ot \cdots \ot \KaEm_{\mathbf{p}_1}$, was to be able to access each component of~$\KaEm$ through the indices $\bce\in\CupBar$ in a natural manner.
We can now express the general element of $\KaEm = \otb_{\bce\in\CupBar} \CB_\bce$ as $\ot_{\bce\in\CupBar} b(-s_{\bce})$ or $\ot_\bce b(-s_{\bce})$, without explicitly writing down all the tensor product components in some clear order.
Note that we are omitting the color subscript $i\in I$ from what should be written in the form~$b_i(-s_{\bce})$.
If required, the subscript may be recovered from the index~$\bce\in\BaCrBar_r$.

With the new indexing scheme, we can finally state the goal of this section.
It is to show that the image of the Kashiwara embedding
\begin{equation}
\Binf \hookrightarrow \Binf \ot \otb_{\bce\in\CupBar} \CB_\bce
\end{equation}
is
\begin{equation}
\Kinf =
\{b_\infty\} \ot
\left\{
\ot_{\bce} b(-s_{\bce})
\,\left\vert\,
\begin{minipage}{17em}
(a) $s_{\bce} \geq 0$ for every $\bce\in\CupBar$;\\
(b) $s_{\bce_1} \leq s_{\bce_2}$ for every $\bce_1, \bce_2\in\BaCrBar_r \subset \CupBar$ such that $\bce_1 \comesfrom \bce_2$.
\end{minipage}
\right.
\right\}.
\end{equation}
This will be done by presenting an explicit injective strict crystal morphism from~$\Tinf$ into $\Binf \ot \otb_{\bce\in\CupBar} \CB_\bce$ whose image is~$\Kinf$.
The claim follows from this, since both $\Kinf$ and the image of the Kashiwara embedding are sub-crystals of the same crystal $\Binf\ot\otb_{\bce\in\CupBar}\CB_\bce$ that are isomorphic to~$\Binf$ and the two share at least one element, namely,
\begin{equation}
b_\infty \ot \big(\ot_{c\in\CupBar} b(0)\big),
\end{equation}
corresponding to the highest weight element~$b_\infty$.

Our claim is that the map $\Tinf \rightarrow \Binf \ot \otb_{\bce\in\CupBar} \CB_\bce$, given by
\begin{equation}\label{eq:mapsto}
\left(\,
\begin{minipage}{15em}
marginally large tableau with reduced set of accumulated box counts $(t_{\bce})_{\bce\in\CupBar}$
\end{minipage}\,
\right)
\mapsto
b_\infty \ot \big(\ot_{\bce\in\CupBar} b(-t_{\bce})\big),
\end{equation}
is a strict crystal morphism that is bijective onto~$\Kinf$.
Note that Proposition~\ref{prop:bijection} already assures us of the bijectivity claim.
All other arguments for a proof of the crystal isomorphism claim being standard, let us just show that this map commutes with the Kashiwara operators.

We will rely on Proposition~\ref{prop:mtpr} and compute the $a_k$'s given there for elements of both~$\Tinf$ and~$\Kinf$.
The simplest case is for the color $i=1$.
The $a_k$ values for the right-hand side element of~\eqref{eq:mapsto} in relation to $\ftil_1$ or $\etil_1$ actions are all $-\infty$, except for the three listed below.
\begin{align}
a_\infty &= 0.\label{eq:tmprhs1}\\
a_{1,\bar{1}5} &= t_{1,\bar{1}5} - t_{1,\bar{2}6}.\label{eq:tmprhs}\\
a_{1,\bar{1}2} &= t_{1,\bar{1}2} - t_{1,\bar{2}6} + 2t_{1,\bar{1}5} - t_{1,\bar{2}15} - t_{1,\bar{2}3}\\
 &= a_{1,\bar{1}5} - (t_{1,\bar{2}15} - t_{1,\bar{1}5}) + (t_{1,\bar{1}2} - t_{1,\bar{2}3}).\label{eq:tmprhs4}
\end{align}
Here, we are using a more natural index in place of the $k$'s, and the three elements from the right-hand side of~\eqref{eq:mapsto} that are reflected here appear in the order
\begin{equation}
b_\infty \ot \cdots \ot b(-t_{1,\bar{1}5}) \ot \cdots \ot b(-t_{1,\bar{1}2}) \ot \cdots,
\end{equation}
regardless of which source-to-sink directed path is chosen to make the tensor product order explicit.
Also note that the right-hand side of~\eqref{eq:tmprhs} is always non-negative.
Hence, according to Proposition~\ref{prop:mtpr}, $\etil_1$ will act on
\begin{equation}
\begin{cases}
b_\infty &\text{if $t_{1,\bar{2}15} - t_{1,\bar{1}5} \geq t_{1,\bar{1}2} - t_{1,\bar{2}3}$ and $t_{1,\bar{1}5} = t_{1,\bar{2}6}$},\\
b(-t_{1,\bar{1}5}) &\text{if $t_{1,\bar{2}15} - t_{1,\bar{1}5} \geq t_{1,\bar{1}2} - t_{1,\bar{2}3}$ and $t_{1,\bar{1}5} > t_{1,\bar{2}6}$},\\
b(-t_{1,\bar{1}2}) &\text{if $t_{1,\bar{2}15} - t_{1,\bar{1}5} < t_{1,\bar{1}2} - t_{1,\bar{2}3}$},
\end{cases}
\end{equation}
with the $\etil_1$ action on $b_\infty$ being zero and the other two cases reducing either $t_{1,\bar{1}5}$ or $t_{1,\bar{1}2}$ by one.
Similarly, $\ftil_1$ will act on
\begin{equation}
\begin{cases}
b(-t_{1,\bar{1}5}) &\text{if $t_{1,\bar{2}15} - t_{1,\bar{1}5} > t_{1,\bar{1}2} - t_{1,\bar{2}3}$},\\
b(-t_{1,\bar{1}2}) &\text{if $t_{1,\bar{2}15} - t_{1,\bar{1}5} \leq t_{1,\bar{1}2} - t_{1,\bar{2}3}$},
\end{cases}
\end{equation}
with the $\ftil_1$ action increasing either $t_{1,\bar{1}5}$ or $t_{1,\bar{1}2}$ by one.

The $a_k$'s appearing in Proposition~\ref{prop:mtpr} are less straightforward to write down for a tableau, because the number of its tensor product components is not fixed and there can even be multiple instances of the same crystal element.
To treat this situation, we write the marginally large tableau as a tensor product of its boxes through the middle eastern reading and add parenthesis at appropriate places to view the long tensor product as a tensor product of a small number of grouped components.
For the $\etil_1$ and $\ftil_1$ actions, the grouping should be done as follows.
\begin{enumerate}
\item $g_1$: tensor product of all boxes in the top row of the tableau containing $\bce\in\BaCr_5$ such that $\bar{5} \comesfrom \bce \comesfrom \bar{2}6$.
\item $g_2$: tensor product of all boxes in the top row of the tableau containing $\bce\in\BaCr_5$ such that $\bar{1}\bar{6}2 \comesfrom \bce \comesfrom \bar{1}5$.
\item $g_3$: tensor product of all boxes in the top row of the tableau containing $\bce\in\BaCr_5$ such that $\bar{6}1 \comesfrom \bce \comesfrom \bar{2}15$.
\item $g_4$: tensor product of all boxes in the top row of the tableau containing $\bce\in\BaCr_5$ such that $\bar{4}2 \comesfrom \bce \comesfrom \bar{2}3$.
\item $g_5$: tensor product of all boxes in the top row of the tableau containing~$\bar{1}2\in\BaCr_5$.
\item $g_6$: tensor product of all boxes in the top row of the tableau containing~$1\in\BaCr_5$.
\item $g_7$: tensor product of all boxes in the forth through first rows of the tableau.
\end{enumerate}
Then, the tensor product form of the given marginally large tableau is equal to
\begin{equation}
g_1 \ot g_2 \ot \dots \ot g_7.
\end{equation}

The $a_k$ values for this view of the tableau in relation to $\etil_1$ and $\ftil_1$ actions are as follows.
\begin{align}
a_1 &= 0\\
a_2 &= t_{1,\bar{1}5} - t_{1,\bar{2}6}\label{eq:sr34fa}\\
a_3 &= a_2\\
a_4 &= a_3 - (t_{1,\bar{2}15} - t_{1,\bar{1}5})\\
a_5 &= a_4 + (t_{1,\bar{1}2} - t_{1,\bar{2}3})\\
a_6 &= a_5\\
a_7 &= a_6 - \big(\nbxr{1}{1} - \nbxr{\bar{1}2}{2}\big)
\end{align}
Note that the right-hand side of~\eqref{eq:sr34fa} is non-negative because $\bar{2}6 \comesfrom \bar{1}5$.
Also note that~$a_7$ is strictly less than~$a_6$, because a large tableau has more boxes containing $1\in\BaCr_5$ on the fifth row than even the total number boxes on its fourth row.

Hence, according to Proposition~\ref{prop:mtpr}, $\etil_1$ will act on
\begin{equation}
\begin{cases}
g_1 &\text{if $t_{1,\bar{2}15} - t_{1,\bar{1}5} \geq t_{1,\bar{1}2} - t_{1,\bar{2}3}$ and $t_{1,\bar{1}5} = t_{1,\bar{2}6}$},\\
g_2 &\text{if $t_{1,\bar{2}15} - t_{1,\bar{1}5} \geq t_{1,\bar{1}2} - t_{1,\bar{2}3}$ and $t_{1,\bar{1}5} > t_{1,\bar{2}6}$},\\
g_5 &\text{if $t_{1,\bar{2}15} - t_{1,\bar{1}5} < t_{1,\bar{1}2} - t_{1,\bar{2}3}$}.
\end{cases}
\end{equation}
The $\etil_1$ action on $g_1$ results in zero.
The $\etil_1$ action on $g_2$ reduces $t_{1,\bar{1}5}$ by one and has no effect on other $t_{r,\bce}$'s, regardless of which tensor product component within~$g_2$ the action falls on.
The $\etil_1$ action on $g_5$ reduces $t_{1,\bar{1}2}$ by one and affects no other $t_{r,\bce}$'s.

Similarly, $\ftil_1$ will act on
\begin{equation}
\begin{cases}
g_3 &\text{if $t_{1,\bar{2}15} - t_{1,\bar{1}5} > t_{1,\bar{1}2} - t_{1,\bar{2}3}$},\\
g_6 &\text{if $t_{1,\bar{2}15} - t_{1,\bar{1}5} \leq t_{1,\bar{1}2} - t_{1,\bar{2}3}$}.
\end{cases}
\end{equation}
The $\ftil_1$ action on $g_3$ will increasing $t_{1,\bar{1}5}$ by one and leave other $t_{r,\bce}$'s unchanged.
The $\ftil_1$ action on $g_6$ will fall on the right-most $1$-box in the top row of~$T$, change it into a $\bar{1}2$-box, and call for an insertion of a new $1$-box.
This will increase $t_{1,\bar{1}2}$ by one and have no effect on other~$t_{r,\bce}$'s.

Comparing the descriptions of $\etil_1$ and $\ftil_1$ on $\Tinf$ and $\Kinf$, we can conclude that the mapping~\eqref{eq:mapsto} commutes with these two Kashiwara operators.
Checking the compatibility of other Kashiwara operator actions with the mapping~\eqref{eq:mapsto} can be done in a case by case manner.

We have thus arrived at our main result.
\begin{theorem}\label{thm:iso}
The map $\Tinf \rightarrow \Kinf \subset \Binf \ot \otb_{\bce\in\CupBar} \CB_\bce$ given by
\begin{equation*}
\left(\,
\begin{minipage}{17em}
marginally large tableau with reduced set of accumulated box counts $(t_{\bce})_{\bce\in\CupBar}$
\end{minipage}\,
\right)
\mapsto
b_\infty \ot \big(\ot_{\bce\in\CupBar} b(-t_{\bce})\big)
\end{equation*}
is an isomorphism of crystals.
\end{theorem}

As was discussed before, the image $\Kinf$ of the above crystal isomorphism is the image of a Kashiwara embedding.

\begin{corollary}\label{cor:KashImbIm}
The image of the Kashiwara embedding
\begin{equation*}
\Binf \hookrightarrow \Binf \ot \otb_{\bce\in\CupBar} \CB_\bce
\end{equation*}
is
\begin{equation*}
\{b_\infty\} \ot
\left\{
\ot_{\bce} b(-s_{\bce})
\,\left\vert\,
\begin{minipage}{17em}
\textup{(}a\textup{)}~$s_{\bce} \geq 0$ for every $\bce\in\CupBar$;\\
\textup{(}b\textup{)}~$s_{\bce_1} \leq s_{\bce_2}$ for every $\bce_1, \bce_2\in\BaCrBar_r \subset \CupBar$ such that $\bce_1 \comesfrom \bce_2$.
\end{minipage}
\right.
\right\}.
\end{equation*}
\end{corollary}

\section{Description of $\Binf$ through Kashiwara Embedding for Type-$E_7$}\label{sec:KE.E7}

All the claims of Section~\ref{sec:KE.E6} were carefully written so that each remains valid for the $E_7$ type.
In particular, the final results Theorem~\ref{thm:iso} and Corollary~\ref{cor:KashImbIm} hold true for the $E_7$ type.

Most of the proofs given in Section~\ref{sec:KE.E6} also carry over to the $E_7$ type with no change.
Below, we provide brief comments for those arguments that require a little bit of extra explanations.

\paragraph{Lemma~\ref{lem:tmin}}

The directed graph $\BaCr_6$ for the $E_7$ type contains one node that receives three incoming arrows.
We clarify that Lemma~\ref{lem:tmin} still applies to any two of the three arrows at such a node.
The proof of Lemma~\ref{lem:tmin} remains valid for the~$E_7$ type, word for word, although the checking of $G_1\cap G_2 = G_0$ requires a 3-dimensional view of the directed graph.

Note that, if $\bce_0 \leftarrow \bce_1$, $\bce_0 \leftarrow \bce_2$, and $\bce_0 \leftarrow \bce_3$, then Lemma~\ref{lem:tmin} implies the weaker statement $t_{\bce_0} = \min(t_{\bce_1}, t_{\bce_2}, t_{\bce_3})$.

\paragraph{Lemma~\ref{lem:capmin}}

During the proof, we located a certain \emph{rectangle} within~$\BaCr_r$.
This may be trivial to do for the $E_6$ type, but is not always so with the $E_7$-type $\BaCr_6$.

Let us first illustrate two $E_7$-type examples, whose rectangles are less straightforward to locate.
The reader should trace these out from the diagrams of Appendix~\ref{app:E7BaCr}.
The nodes $\bar{2}67$ and $\bar{6}1$ cannot be placed on a directed path of~$\BaCr_6$ for type $E_7$, and the (bent) rectangle for this pair would be as follows.
\begin{center}
\includegraphics{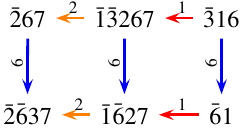}
\end{center}
For the $\bar{1}6$-$\bar{6}1$ node pair, either of the following two rectangles could be used in the proof.
\begin{center}
\includegraphics{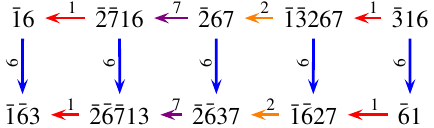}
\qquad\qquad
\includegraphics{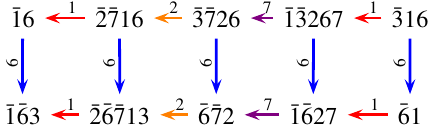}
\end{center}

There are two cases with the $E_7$ type $C_6$ for which our rectangle argument is invalid.
These are the $\bar{2}67$-$\bar{6}\bar{7}2$ pair and the $\bar{3}\bar{7}26$-$\bar{2}\bar{6}37$ pair cases.
However, for both of these cases, we can visually check through the directed graph of Appendix~\ref{app:E7BaCr} that
\begin{equation}
G_1\cap G_2 = \{x\in\BaCrCirc_r\mid x\comesfrom \bar{2}\bar{6}\bar{7}13 \},
\end{equation}
and we can also directly check that
\begin{alignat}{3}
s_{\bar{2}\bar{6}\bar{7}13}
&= \min(s_{\bar{2}\bar{7}16}, s_{\bar{6}\bar{7}2})&
&= \min(s_{\bar{2}67}, s_{\bar{3}\bar{7}26}, s_{\bar{6}\bar{7}2})&
&= \min(s_{\bar{2}67}, s_{\bar{6}\bar{7}2}),\\
s_{\bar{2}\bar{6}\bar{7}13}
&= \min(s_{\bar{6}\bar{7}2}, s_{\bar{2}\bar{6}37})&
&= \min(s_{\bar{3}\bar{7}26}, s_{\bar{1}\bar{6}27}, s_{\bar{2}\bar{6}37})&
&= \min(s_{\bar{3}\bar{7}26}, s_{\bar{2}\bar{6}37}).
\end{alignat}
Hence, Lemma~\ref{lem:capmin} remains valid even for these two exceptional cases.

\paragraph{Lemma~\ref{lem:ABCsurj}}

During the proof, to show that each $n_\bce$ is non-negative, we discussed the cases of~$\bce$ having out-degrees 1 and~2.
The $E_7$-type directed graph $\BaCr_6$ contains one node of out-degree 3.
The argument we gave for the out-degree 2 case applies to any two of the arrows from the out-degree 3 node and this is sufficient to show that $n_\bce$ is non-negative.

\paragraph{Proposition~\ref{prop:bijection}}

Unlike the $\BaCr_r$'s for the $E_6$ type whose nodes are always of in-degree 0, 1, or 2, the directed graph $\BaCr_6$ for the $E_7$ type contains a single node of in-degree 3, namely, the node~$\bar{2}\bar{6}\bar{7}13$.
Hence, we must check if the inductive extension of~$\bar{\mathbf{s}}$ to~$\mathbf{s}$ done in the proof did not call for assignments of contradicting values to~$s_{\bar{2}\bar{6}\bar{7}13}$.
However, it is easy to see that all three possibilities
\begin{align}
\min(s_{\bar{2}\bar{7}16}, s_{\bar{2}\bar{6}37})
 &= \min\big(
        \min(s_{\bar{2}67}, s_{\bar{3}\bar{7}26}),
        \min(s_{\bar{2}67}, s_{\bar{1}\bar{6}27})
        \big),\\
 &= \min(s_{\bar{2}67}, s_{\bar{3}\bar{7}26}, s_{\bar{1}\bar{6}27})\\
\min(s_{\bar{6}\bar{7}2}, s_{\bar{2}\bar{7}16})
 &= \min\big(
        \min(s_{\bar{3}\bar{7}26}, s_{\bar{1}\bar{6}27}),
        \min(s_{\bar{2}67}, s_{\bar{3}\bar{7}26})
        \big),\\
 &= \min(s_{\bar{3}\bar{7}26}, s_{\bar{1}\bar{6}17}, s_{\bar{2}67})\\
\min(s_{\bar{2}\bar{6}37}, s_{\bar{6}\bar{7}2})
 &= \min\big(
        \min(s_{\bar{2}67}, s_{\bar{1}\bar{6}27}),
        \min(s_{\bar{3}\bar{7}26}, s_{\bar{1}\bar{6}27})
        \big)\\
 &= \min(s_{\bar{2}67}, s_{\bar{1}\bar{6}27}, s_{\bar{3}\bar{7}26})
\end{align}
for the definition of~$s_{\bar{2}\bar{6}\bar{7}13}$ lead to the same value.

The fact that these three values are the same further implies that two of the three values $s_{\bar{2}\bar{7}16}$, $s_{\bar{2}\bar{6}37}$, and $s_{\bar{6}\bar{7}2}$ are the same and that the above common minimum value is also equal to
\begin{equation}
\min(s_{\bar{2}\bar{7}16}, s_{\bar{2}\bar{6}37}, s_{\bar{6}\bar{7}2}).
\end{equation}
This observation allows us to adjust the final part of the proof to hold for the $E_7$ type.

\appendix
\section{Basic Crystals}\label{app:BaCr}

The two basic crystals that were used extensively in this paper are given here explicitly.
These were taken from~\cite{JonSch2010,HonLee12}.
The differences in labeling of elements and arrows are due to the difference in Dynkin diagram labeling, and the rearrangement of elements were done to make the concepts used in our paper more visible.

\subsection{Basic Crystal for $E_6$ type}
\label{app:E6BaCr}

The basic crystal $\BaCr = \CB(\La_1)$ for the $E_6$-type is as follows.
\begin{center}
\includegraphics{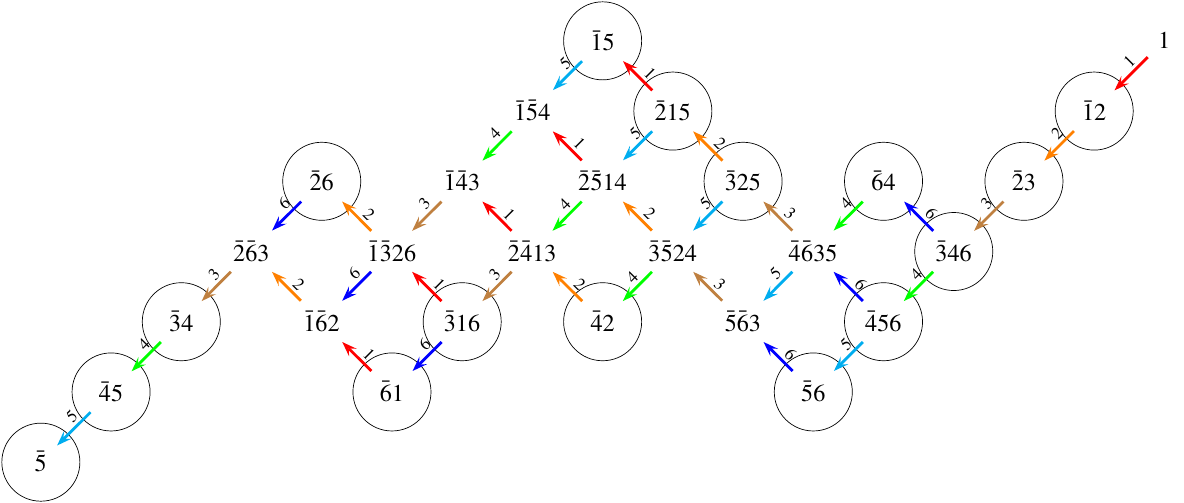}
\end{center}
This is also~$\BaCr_5$, the set of elements that may appear on the fifth (or the top) row of a type-$E_6$ large tableau.
The circled nodes are the crystal elements that receive a single incoming arrow and these form the subset~$\BaCrBar_5$.

The set $\BaCr_4$ of basic crystal elements that may appear on the fourth row of a type-$E_6$ marginally tableau is as follows.
\begin{center}
\includegraphics{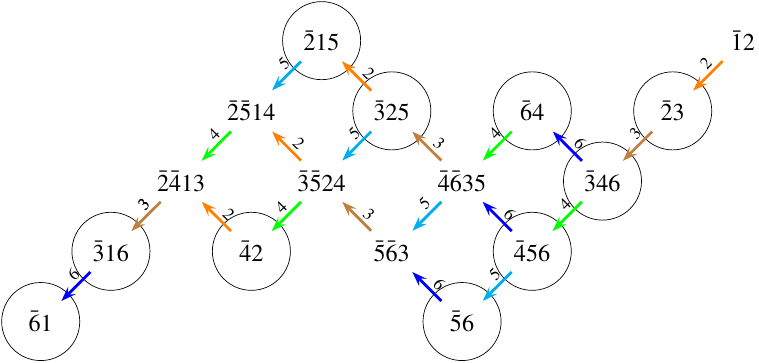}
\end{center}
This directed graph is the largest component from what remains when all the $1$-arrows are removed from the basic crystal~$\BaCr$.
The circled nodes form the smaller set $\BaCrBar_4$ of crystal elements that receive a single incoming arrow.

The directed graphs $\BaCr_{3}$, $\BaCr_{2}$, and $\BaCr_{1}$ are as follows.
\begin{center}
\includegraphics{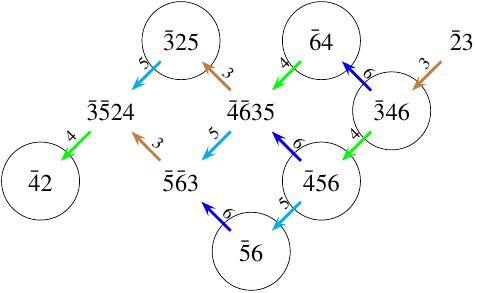}\qquad
\includegraphics{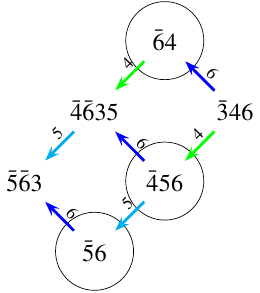}\qquad
\includegraphics{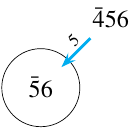}
\end{center}
The circled nodes from  these directed graphs form $\BaCrBar_3$, $\BaCrBar_2$, and~$\BaCrBar_1$.

\subsection{Basic Crystal for $E_7$ type}\label{app:E7BaCr}

The basic crystal $\BaCr = \CB(\La_7)$ for the $E_7$-type is as follows.
\begin{center}
\includegraphics{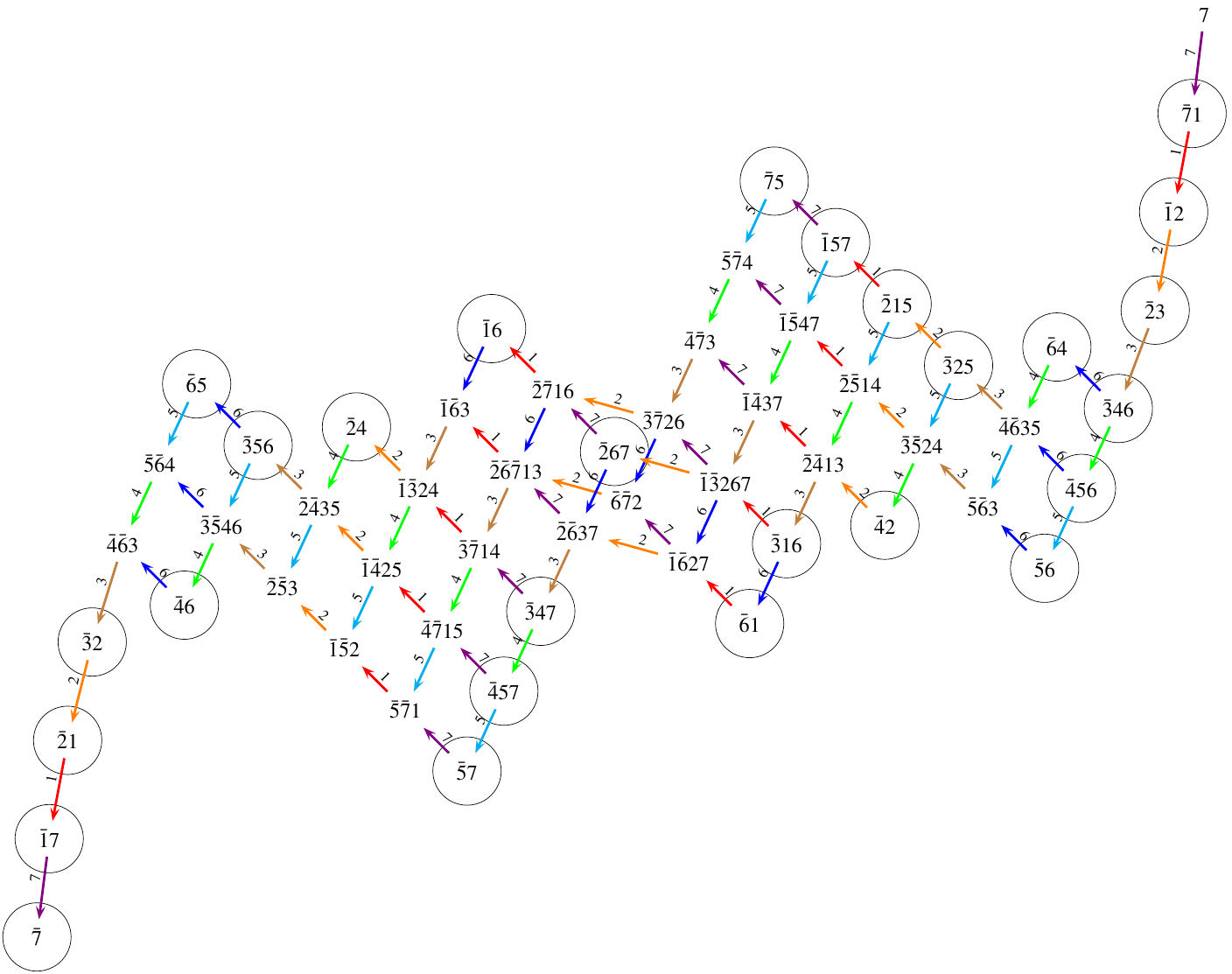}
\end{center}
This is also~$\BaCr_6$, the set of elements that may appear on the sixth (or the top) row of a type-$E_7$ large tableau.
The circled nodes form the subset~$\BaCrBar_6$.

The directed graph $\BaCr_5$ is as follows.
\begin{center}
\includegraphics{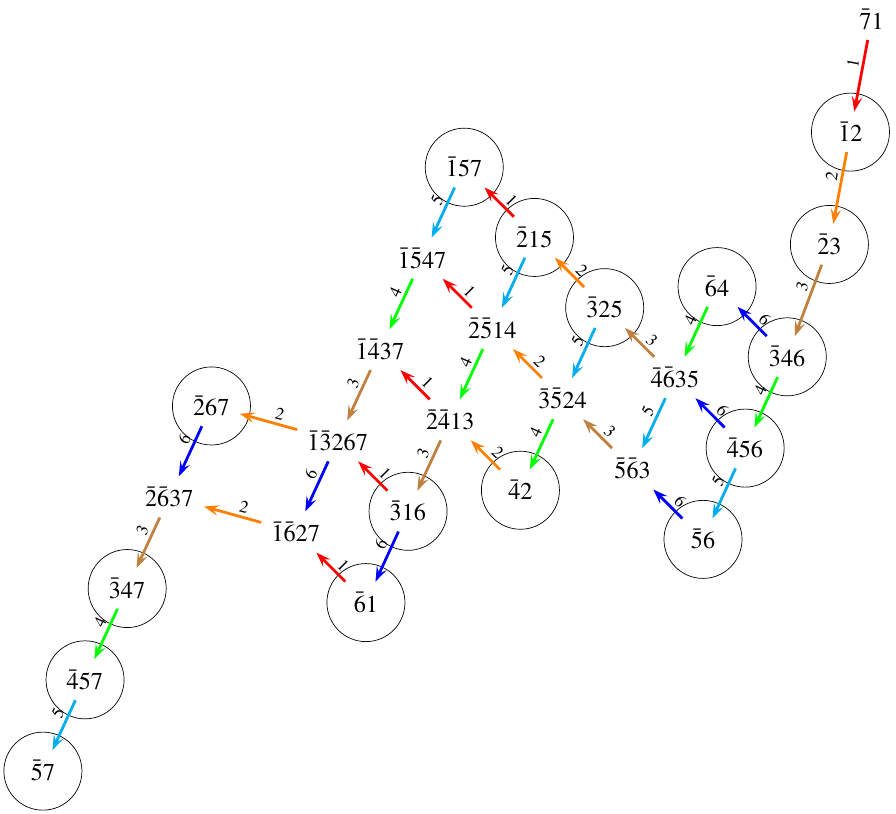}
\end{center}

The directed graphs $\BaCr_4$, $\BaCr_3$, $\BaCr_2$, and $\BaCr_1$ for the $E_7$ type are all identical to those for the $E_6$ type given in Appendix~\ref{app:E6BaCr}.



\end{document}